\documentclass[reqno,10pt]{amsart}

\usepackage{graphicx}

\usepackage{amssymb}
\usepackage{amsmath}
\usepackage{amsfonts}
\usepackage{enumerate}

\newcommand{\mcal}{\mathcal}
\newcommand{\q}{\left\{}
\newcommand{\w}{\right\}}
\newcommand{\re}{\mathbb{R}}

\newcommand{\iz}{\mathbb{Z}}
\newcommand{\nn}{\mathbb{N}}

\newcommand{\ra}{\rightarrow}
\newcommand{\tb}{\textbf}

\newcommand{\ti}{\textit}

\newcommand{\ap}{\alpha}
\newcommand{\bt}{\beta}

\newcommand{\sbs}{\subseteq}

\newcommand{\li}{\langle}
\newcommand{\ri}{\rangle}

\newcommand{\bsk}{\bigskip}

\newcommand{\opn}{\operatorname}
\newcommand{\rank}{\opn{rank}\,}
\newcommand{\card}{\opn{card}\,}
\newcommand{\supp}{\opn{supp}\,}
\newcommand{\ran}{\opn{range}\,}

\newcommand{\av}{\mathcal{V}_{\beta}}

\newcommand{\rf}{\Lambda}
\newcommand{\lp}{\left(}
\newcommand{\rp}{\right)}
\newcommand{\lb}{\left[}
\newcommand{\rb}{\right]}

\newcommand{\vv}{ {\bf v}}
\newcommand{\vx}{ {\bf x}}
\newcommand{\vi}{ {\bf i}}
\newcommand{\vw}{ {\bf w}}

\theoremstyle{plain}
\newtheorem{theorem}{Theorem}[section]
\newtheorem{proposition}[theorem]{Proposition}
\newtheorem{lemma}[theorem]{Lemma}

\theoremstyle{definition}
\newtheorem{definition}[theorem]{Definition}
\newtheorem{example}[theorem]{Example}

\theoremstyle{remark}

\begin{document}

\title{Any  multi-index  sequence has an interpolating measure}

\author[H. Choi]{Hayoung Choi}
\address{Hayoung Choi \\  Department of Mathematics \\ Kyungpook National University \\ Daegu, Republic of Korea 41566}            
             \email{hayoung.choi@knu.ac.kr}
\thanks{The first-named author was supported by Basic Science Research Program through the National Research Foundation of Korea (NRF) funded by the Ministry of Education (20**).}

\author[S. Yoo]{ Seonguk Yoo}
\address{Seonguk Yoo \\  Department of Mathematics Education and RINS \\ Gyeongsang National University \\ Jinju, Republic of Korea 52828}
\email{seyoo@gnu.ac.kr}

\thanks{The second-named author was supported by Basic Science Research Program through the National Research Foundation of Korea (NRF) funded by the Ministry of Education  (2020R1F1A1A01070552).}

\date{8, 20, 2020}



\begin{abstract} R. P. Boas showed that any single-index sequence $\q \bt_i \w_{i=0}^\infty$ of real numbers can be represented as $\bt_i =\int_0^\infty x^i \, d\mu$ ($i=0,1,2,\ldots$), where $\mu$ is a signed measure.  As Boas said his observation seemed to be quite unexpected; however, it is even possible  to extend the result to any multi-index  sequence of real numbers.   In addition,  we can also prove that any multi-index finite  sequence admits a measure of a similar type.  

\end{abstract}

\maketitle

\section{Introduction}\label{intro}

We first discuss finite sequences and then introduce the result of infinite sequences. Let $\beta \equiv\beta^{(m)}= \q  \beta_{\tb i} \in \re:  \tb i \in \iz^d_+, \ |\tb i| \leq m  \w$, with $\beta_{\bf 0}\neq 0$,  
be  a $d$-dimensional multisequence of  degree $m$.  It is called a \ti{truncated moment sequence}. 
For a closed set $K\sbs \re^d$, the  \ti{truncated $K$-moment problem} (TKMP) entails finding necessary and sufficient conditions for the existence of  a positive Borel measure $\mu$ on $\re^d$ with $\supp \mu \sbs K$ such that
\begin{eqnarray*}
\beta_{\tb i}=\int \tb x^{\tb i} \,\, d\mu(\tb x) \,\,\,(\vi \in \mathbb Z^d_+, \ |\vi|\leq m),
\end{eqnarray*}
where $\vx\equiv (x_1,\ldots,x_d)$, $\vi\equiv(i_1,\ldots,i_d)\in \iz^d_+$, and $\vx^\vi:=x_1^{i_1}\cdots x_d^{i_d}$. 
The measure $\mu$ is said to be a $K$-\textit{representing measure} for $\beta$. For the typical case $K=\re^d$, the problem is referred to as  the \ti{truncated real moment problem} (TRMP) and $\mu$ is called simply a \ti{representing measure}. 

In a similar way, we consider the full moment problem for an infinite sequence 
$\beta \equiv\beta^{(\infty)}= \q  \beta_{\tb i} :  \tb i \in \iz^d_+ \w$. As well known by H. L. Hamburger  for $d=1$, the sequence has a representing measure supported on $\re$ if and only if the Hankel matrice, $[\beta_{i+j}]_{0\leq i\leq k,\  0\leq j\leq k}$ is positive semidefinite (or simply, positive).  Furthermore, T. J.  Stieltjes  showed that  the single-index  sequence has a representing measure supported in $[0,\infty)$ if and only if both Hankel matrices, $[\beta_{i+j}]_{0\leq i\leq k,\  0\leq j\leq k}$ and $[\beta_{i+j+1}]_{0\leq i\leq k,\  0\leq j\leq k}$ for $k\geq 0$, are positive. 


When $m=2n$, we define   a \textit{moment matrix} $M_d(n) $ of $\bt \equiv \beta^{(2n)}$ as 
$$
M_d(n)\equiv {M}_d(n)(\beta):=(\beta_{\textbf{ i} +\textbf{j} })_{\textbf{ i}, \, \textbf{j}\in \mathbb Z^d_+: \ |\textbf{i}|,\, |\textbf{j}|\leq n}.
$$
Some  properties of $M_d(n)$ have been  important factors for the existence of a representing measure for $\bt$; for example, $M_d(n)$ is necessarily positive (obviously the positivity of $M_d(n)$ is sufficient for $d=1$ but  not sufficient for $d\geq 2$ as well known).  
R. Curto and L. Fialkow have established many elegant results for various moment problems based on a positive  extension of $M_d(n)$. They also have used the functional calculus in the column space of $M_d(n)$; to introduce the functional calculus, we label the columns and rows of  $M_d{(n)}$ with monomials  $X^\vi:=X_1^{i_1}\cdots X_d^{i_d}$ in the  degree-lexicographic order. 
Note that each block with the moments of the same order  in ${M}_d(n)$ is Hankel and that  $M_d(n)$ is symmetric.  In  addition, one can define a sesquilinear form: for ${\bf i, j}\in \iz_+^d$, 
$$
\li X^{\vi} , X^{\bf j}\ri_{M_d(n)}:=\li M_d(n) \widehat{ X^{\bf i}},  \widehat{X^{\bf j}}\ri = \beta_{\bf i+ j},
$$
where $ \widehat{ X^\vi }$ is the column vector associated to the monomial ${ X^\vi}$. 

For a motivation of the main result, let us consider the basic  Fibonacci sequence. 
In particular, take the first six moments and write them as a 2-dimensional moment sequence $
\beta:\q \beta_{00}, \beta_{10},\beta_{01},\beta_{20},\beta_{11},\beta_{02} \w $$= \q1, 1, 2, 3, 5, 8\w$. 
Since $M_2(1)(\bt)$ is not positive, $\bt$ does not admit a representing measure. However,  one can find a  formula to express $\bt$,  such as 
$\bt_{ij} = 1\cdot \lp \frac{1}{2}\rp^i (1)^j + \frac{1}{7} \cdot \lp\frac{9}{2} \rp^i (7)^j
-\frac{1}{7} \cdot (1)^i  (0)^j$;  
that is, there is a signed measure $\mu= 1\cdot \delta_{\lp \frac{1}{2}, 1\rp} + \frac{1}{7} \cdot \delta_{\lp\frac{9}{2}, 7 \rp}
-\frac{1}{7} \cdot \delta_{(1,0)}$ for $\bt$ supported in the plane.     
The coefficients in the formula of the measure are called \ti{densities} and the points are \ti{atoms}  of the  measure.  
This example shows that even though a sequence has no representing measure, it may have a {signed} measure so that  some of the densities might be  negative.      
We define such a measure as an \ti{interpolating} measure $\mu$  for  $\bt$ (finite or infinite) as a, not necessarily positive, Borel measure  such that $\beta_{\tb i}=\int \tb x^{\tb i} \,\, d\mu(\tb x),  \ \vi \in \mathbb Z^d_+$.)



Due to the {Jordan decomposition theorem}, every  interpolating    measure $\mu$ has a decomposition, $\mu  = \mu^+ -\mu^-$ of two positive measures $\mu^+$ and $\mu^-$, at least one of which is finite. 
Interpolating measures appear in many  scientific fields.  For example, they are useful to represent electric charge; the moment problem about a signed measure is related to quantum physics as in \cite{FBW}. Furthermore, there is a possibility that Gauss-Jacobi quadratures would be  generalized through  moment sequences with a signed measure (see \cite{KoRe}).


For $d=1$, 
 R. P. Boas 
showed that any single-index ``infinite'' sequence of real numbers admits an interpolating measure supported in $[0,\infty)$; that is, one can always find a  measure for any sequence of the form $\mu  = \mu^+ -\mu^-$  such that  both  $\mu^+$ and $\mu^-$ are positive Borel measures  supported in $[0,\infty]$ \cite{Boa}.  
Moreover, G. Flessas, K. Burton, and R. R. Whitehead  found an algorithm to find such a measure supported in the real line for a ``finite'' real sequence $\q s_j\w_{j=0}^{2n-1}$ \cite{FBW}. 
As a generalization of these results,  we will see that any finite or infinite sequence has an interpolating measure supported in $\re^d$ for any $d\geq 2$. 
Notice that since moment problems about finite sequences are known to be more general than problems about infinite sequences, we need a solution to each of these two problems.


We conclude this section with another application of the moment problem to the numerical integration. For more details, readers can refer to \cite{GoMe}.  

\begin{definition}
A \tb{quadrature} (or \tb{cubature})  rule of size $p$ and precision $m$ is a numerical integration formula which uses $p$ nodes, is exact for all polynomials of degree at most $m$, and fails to recover the integral some polynomial of degree $m+1$.
\end{definition}

\begin{example}\tb{[Gaussian Quadrature; size $n$, precision $2n-1$]} 
We would like to find nodes $t_0, t_1\ldots, t_{n-1}$ satisfying 
\begin{equation}\label{e-gq0}
\int_{-1}^1 f(t) \, dt =\sum_{j=0}^{n-1} \rho_j f\lp t_j\rp 
\end{equation}
for every polynomial  $f  \text{ with }\deg f\leq 2n-1$. 
Now, we consider interpolating equations with polynomials and we get 
\begin{equation}\label{e-gq}
\sum_{j=0}^{n-1} \rho_j t_j^k =\int_{-1}^1 t^k \, dt  =
\begin{cases}
0 &  k=1,3,\ldots,2n-1;\\
\frac{2}{k+1} &k=0,2,\ldots,2n-2. 
\end{cases} 
\end{equation} 
If  $n=2$,  (\ref{e-gq}) becomes the system of polynomial equations 
$$\q
\begin{array}{cl}
\rho_0+\rho_1 &= \ 2;\\
\rho_0 t_0+\rho_1 t_1&=\ 0;\\
\rho_0 t_0^2+\rho_1 t_1^2&=\ 2/3;\\
\rho_0 t_0^3+\rho_1 t_1^3&=\ 0.
\end{array}\right.
$$
The solution is $\rho_0=\rho_1=1$, $t_0=-1/\sqrt{3}$, and $t_1=1/\sqrt{3}$.  Thus we easily see 
\begin{align*}
\int_{-1}^1  \lp  a_0+ a_1 t  \right.&\left. +a_2 t^2+ a_3 t^3 \rp \, dt  \\
&=   a_0 (\rho_0+\rho_1) + a_1 (\rho_0 t_0+\rho_1 t_1 )+ a_2 \lp \rho_0 t_0^2+\rho_1 t_1^2 \rp+ a_3 \lp \rho_0 t_0^3+\rho_1 t_1^3 \rp \\
& =\int_{-1}^1  \lp a_0 + a_1 t +a_2 t^2+ a_3 t^3\rp \, d\mu, 
\end{align*}
where $\mu:=\rho_0 \delta_{t_0}+\rho_1 \delta_{t_1}$. This solution  in numerical analysis textbooks  is usually based on  {Legendre polynomials}. 
With an approach via the truncated moment problem, we can find an alternative solution as follows:  Let $\beta_0:=2, \beta_1:=0, \beta_2:=2/3, \beta_3:=0$ and form a Hankel matrix $H$ with a parameter $\ap$,  
\begin{align*}
H:=\begin{pmatrix}
\beta_0 & \beta_1 & \beta_2\\
\beta_1 & \beta_2 & \beta_3\\
\beta_2 & \beta_3 & \ap \\
\end{pmatrix}
=&\begin{pmatrix}
2 & 0 & 2/3\\
0 & 2/3 & 0\\
2/3 & 0 & \ap \\
\end{pmatrix}
\end{align*}
For the sake of a minimal number of nodes, we want $\rank H=2$; thus, $\ap=2/9$. After labeling the columns in $H$ as $\ti{1}, T, T^2$, the column relation in $H$ can be written as $T^2=(1/3)\ti{1}$. In \cite{CF91}, it is known the roots of the equation $t^2=1/3$  (that is, $t_0=-1/\sqrt{3}$ and $t_1=1/\sqrt{3}$)  are the nodes. We may compute the densities by solving the Vandermonde equation: 
$$
\begin{pmatrix}\
1& 1 \\
 t_0 & t_1\\
 t_0^2 &   t_1^2\\
 t_0^3 &   t_1^3
\end{pmatrix} 
\begin{pmatrix}
\rho_0 \\ \rho_1
\end{pmatrix}
=
\begin{pmatrix}
\beta_0 \\ \beta_1 \\ \beta_2 \\ \beta_3
\end{pmatrix}, 
$$
whose solution is obviously $\rho_0=\rho_1=1.$ 
\end{example}
This method seems to provide an economical way to solve a qudrature problem and we will  see  the main result of this article gives  a technique for more general cases, that is, when a signed measure arises in (\ref{e-gq0}).

\section{The Consistency and Rank-one Decompositions of Moment Matrices}

This Section is designed to introduce some background knowledge for dealing with truncated moment sequences. 

\subsection{The consistency} 
We are about to define an algebraic set associated to $M_d(n)$. Let $\mcal P:=\re[x_1,\ldots,x_d]$ and let $\mcal P_k:=\q p\in \mcal P: \deg p \leq k \w$.  Since we labeled columns in $M_d(n)$ with monomials, a column relation in $M_d(n)$ can be written as $p(\bf X)=\tb 0$ for some $p\in \mcal P_n$. 
Let $\mathcal{Z}(p)$ denote the zero set of a polynomial $p$ and we define the \textit{algebraic variety} $\av$ of $\beta $ or $M_d(n)$  by
\begin{equation} \label{d-variety}
\mathcal{V_\beta} \equiv  \mcal{V}_{M_d(n)}:=\bigcap_{p ({\bf X})=\tb 0 }\; \mathcal{Z}
(p).
\end{equation}

Given $\beta\equiv\beta^{(m)}$, define the \ti{Riesz functional} $\rf\equiv\rf_\bt : \mcal P_m \ra \re$ by 
$\rf\lp \sum a_\vi \vx^\vi \rp:=\sum a_\vi \bt_\vi$. 
We also define a notion which is the key to the main result of this note;  $\bt\equiv\bt^{(2n)}$ or $M_d(n)\equiv M_d(n)(\bt^{(2n)})$ is said to be $V$-\textit{consistent} for a set $V\in \re^d$ if  the following holds:
\begin{equation}\label{cs}
   p\in\mathcal P_{2n},~p|_{  V}\equiv0 \Longrightarrow \Lambda(p)=0.
\end{equation}

This is a property of the moment sequence that guarantees the existence of an interpolating measure. Here is a formal result:
\begin{lemma}\label{r-cons}  \cite[Lemma 2.3]{tcmp11}
Let $L:\mcal P_{2n} \ra \re$ be a linear functional and let $ V\sbs \re^d$. Then the following statements are equivalent:
\begin{enumerate}[(i)]
\item \label{r-cons1} There exist $\ap_1,\ldots,\ap_\ell \in \re$ and there exist ${\bf w}_1, \ldots,{\bf w}_\ell\in V$ such that  for all  $p\in \mcal P_{2n}$
\begin{equation}\label{eq1}
L(p)= \sum_{k=1}^\ell \ap_k p({\bf w}_k) .
\end{equation}

\item \label{r-cons2} If $p\in \mcal P_{2n}$ and $p|_{V}\equiv 0$, then $L(p)=0$.
\end{enumerate}
\end{lemma}

If $L$ is the Riesz functional of the moment sequence $\beta$, then Lemma \ref{r-cons}(\ref{r-cons2}) is just as   the $\av$-consistency condition of $\bt$ and $ \sum_{k=1}^\ell \ap_k \delta_{\vw_k}$ is an interpolating measure for $\bt$.   
While it seems like Lemma \ref{r-cons} gives a concrete solution for $\bt$ to have an interpolating measure,   
we should indicate that checking the consistency is a highly nontrivial process.  To show that $\bt$ is $V$-consistent, it is essential (but, difficult) to find a representation of  all the polynomials  vanishing on $V$.

For $M_d(n)$ to have a (positive) representing measure, $\bt$ must be $\av$-consistent; in the \ti{extremal} cases (that is, $\rank M_d(n)=\card \av$), it is known that  $M_d(n)(\bt)$ is consistent if and only if $\bt$ admits a unique $\rank M_d(n)$-atomic representing measure whose support is exactly $\av$\cite{tcmp11}.

In particular, when a positive $M_d(n)$ is invertible,  we know $\av=\re^d$ and the only polynomial vanishing on $\re^d$ is the zero polynomial. Thus, $M_d(n)$  is naturally consistent and has an interpolating measure.

\subsection{Rank-one decompositions}


After rearranging the terms in (\ref{eq1}) by the sign of densities,  we write a  measure $\mu$ for a consistent $M_d(n)$ as
\begin{equation}\label{eq2}
\mu= \sum_{k=1}^{s} \ap_k \delta_{{\bf w}_k} - \sum_{k=s+1}^{\ell} \ap_k \delta_{{\bf w}_k},
\end{equation}
where $\ap_k>0$ for all $k=1,\ldots,\ell$; we denote the first summand in (\ref{eq2}) as $\mu^+$ and the second as $\mu^-$. 
Due to this fact,  a bound of the cardinality of the support of an interpolating measure is established:

\begin{proposition} A minimal interpolating measure for a consistent $M_d(n)$ is at most  $(2n+1)(2n+2)$-atomic. 
\end{proposition}

\begin{proof}
If  $M_d(n)$ is consistent with a measure $\mu  = \mu^+ -\mu^-$ of two positive finitely atomic measures $\mu^+$ and $\mu^-$,  we may write $M_d(n)=M[\mu^+]-M[\mu^-]$, where each term  is a moment matrix generated by the corresponding measure of the same size as $M_d(n)$. 
A result \cite[Theorem 2]{BaTe} by C. Bayer and J. Teichmann showed that the cardinality of the support of a positive  measure is at most $\dim \mcal P_{2n}$ in the presence of a representing measure for a moment matrix associated to a moment sequence of degree $2n$.

Since $M[\mu^+]$  and $M[\mu^-]$ have a positive measure, it follows that a minimal measure for each moment matrix is at most $\dim \mcal P_{2n}$-atomic.  
Therefore, we conclude that the cardinality of a minimal interpolating measure is at most $2 (\dim \mcal P_{2n})=(2n+1)(2n+2)$.
\end{proof}

Many solutions of TRMP for a positive measure depend on finding a positive moment matrix extension of $M_d(n)$. However, this approach needs to allow new parameters and constructing an extension is not handy for most cases when $n\geq 3$. Alternatively, R. Curto and the second-author recently have used a decomposition of $M_d(n)$ for the study of TRMP.  To introduce the decomposition, we now  define some notations: Let ${\bf w}=(w_1, \ldots, w_d)\in \re^d$ and let   
\begin{enumerate} [(i)] 
\item ${\bf v}({\bf w}):=
\begin{pmatrix}
1\hspace{1.5mm} w_1\hspace{1.5mm}  \cdots\hspace{1.5mm}  w_d\hspace{1.5mm}   w_1^2\hspace{1.5mm}  w_1 w_2\hspace{1.5mm}   w_1 w_3\hspace{1.5mm}   \cdots\hspace{1.5mm}  w_{d-1} w_d \hspace{1.5mm}  w_d^2\hspace{1.5mm}  \cdots\hspace{1.5mm}   w_1^n\hspace{1.5mm}  \cdots  \hspace{1.5mm}   w_d^n
\end{pmatrix}$, 
which is a row vector corresponding to the monomials ${\bf w}^{\vi}$ in the degree-lexicographic order. 

\item $P({\bf w}):={\bf v}({\bf w})^T {\bf v}({\bf w})$, which is indeed the rank-one moment matrix generated by the measure $\delta_{\bf w}$. 
\end{enumerate}
For example, if $d=n=2$ and ${\bf w}=(a,b)$, then 
$$
P({\bf w})=\left(
\begin{array}{cccccc}
 1 & a & b & a^2 & a b & b^2 \\
 a & a^2 & a b & a^3 & a^2 b & a b^2 \\
 b & a b & b^2 & a^2 b & a b^2 & b^3 \\
 a^2 & a^3 & a^2 b & a^4 & a^3 b & a^2 b^2 \\
 a b & a^2 b & a b^2 & a^3 b & a^2 b^2 & a b^3 \\
 b^2 & a b^2 & b^3 & a^2 b^2 & a b^3 & b^4 \\
\end{array}
\right).
$$
Thus, if $M_d(n)$ has an interpolating measure $\mu$ supported in a set $\q {\bf w}_1, \ldots, {\bf w}_\ell \w$, then one should be able to write $M_d(n)=\sum_{k=1}^\ell d_k P({\bf w}_k)$ for some $d_1,\ldots,d_\ell\in \re\setminus \q 0 \w$.

\section {Main Result}

We will verify that any truncated moment matrix turns out to be $\re^d$-consistent after applying proper perturbations, and so it admits an interpolating measure.  To prove the main result, we begin with auxiliary results:

\begin{lemma} \cite{HoJo}\label{r-HJ1}
Assume $A$ and $B$ are matrices of the same size. Then 
$\rank(A+B) =\rank A + \rank B$ if and only if $\ran A \cap \ran B=\q \bf 0\w$ and $\ran A^T \cap \ran B^T=\q \bf 0\w$.
\end{lemma}

As a special case of Lemma \ref{r-HJ1}, one can easily prove:

\begin{lemma}\label{r-ran} 
Assume $A$ and $B$ are Hermitian matrices of the same size and $\rank B=1$. Then 
$\rank(A+B) =1 + \rank A$ if and only if $\ran A \cap \ran B=\q \bf 0\w$.
\end{lemma}

We are ready to introduce a crucial lemma: 

\begin{lemma}\label{r-rann}
 A point $ \vw$ is in $\mcal V_{M_d(n)}$ if and only if the vector $\vv (\vw)$ is in $\ran M_d(n)$. 
\end{lemma}
\begin{proof}
Assume that $\q p_k({\bf X})\equiv \sum a_{\vi}^{(k)} {\bf X}^\vi \w_{k=1}^\ell$ is the set of  polynomials obtained from  column relations in $M_d(n)$. Note that $\opn{span}\, \left\{ \widehat {p_k}\right\}_{k=1}^\ell =\ker M_d(n)$. Now observe:   
\begin{eqnarray*}
\vw \in \mcal V_{M_d(n)} 
&\iff& p_k(\vw)= {\bf 0} \qquad \, \hspace{6mm} \text{for } k=1,\ldots,\ell \\
&\iff&  \sum  a_{\vi}^{(k)} \vw^\vi= {\bf 0}  \hspace{0.2mm}\qquad \text{for } k=1,\ldots,\ell \\
&\iff& \li \widehat{p_k},\vv(\vw) \ri =0 \hspace{7.3mm}\text{for } k=1,\ldots,\ell \\
&\iff& \widehat{p_k} \perp \vv(\vw) \qquad \, \hspace{6mm}\text{for } k=1,\ldots,\ell \\
&\iff& \vv(\vw)\in (\ker M_d(n))^\perp = \ran M_d(n).
\end{eqnarray*}
\end{proof}

\begin{theorem}\label{r-main0}
Any truncated moment sequence $\bt\equiv\beta^{(2n)}$ of degree $2n$ has an interpolating  measure in $\re^d$ for any positive $d\in \iz_+$. 
\end{theorem}
\begin{proof}

Pick a point $\vw_1\in \re^d\setminus\av$. 
Then we know from Lemma \ref{r-rann} that $\vv(\vw_1) \not\in \ran M_d(n)(\bt)$. 
Since $\ran P(\vw_1) = \q \ap \vw_1 : \ap \in \re\w$, it holds that $\ran M_d(n)(\bt) \cap \ran P(\vw_1) = \q \tb 0\w$. 
Therefore, it follows  from Lemma \ref{r-ran}  that 
$\rank (M_d(n)(\bt)+P(\vw_1))=1+\rank M_d(n)(\bt)$. 
Next, choose a point $\vw_2$ which not in the algebraic variety of $M_d(n)(\bt)+P(\vw_1)$ and we know from the same argument that $\rank (M_d(n)(\bt)+P(\vw_1)+P(\vw_2))=2+\rank M_d(n)(\bt)$. 
Keep this process until we obtain an invertible matrix $\widetilde M:=M_d(n)(\bt)+ \sum_{k=1}^\ell P(\vw_k)$ for some $\ell$.  
$\widetilde M$ is naturally consistent, and so it admits an interpolating measure, say $\tilde \mu$. 
Thus, $M_d(n)(\bt)$ has an interpolating measure of the form  $\tilde \mu -  \sum_{k=1}^\ell \delta_{\vw_k}$.

\end{proof}

\begin{theorem} \label{r-main1}
Any  finite sequence has an interpolating  measure. 
\end{theorem}
\begin{proof}
It suffices to cover the cases when  the given sequence is not the type of $\beta^{(2n)}$. Such a sequence cannot fill up the associated moment matrix, so we use new parameters to complete the moment matrix. If it is possible to make the moment matrix invertible, then the extended moment sequence is consistent. Thus, the given sequence has an interpolating measure. 
Otherwise, one can  follow the same process in the proof of Theorem \ref{r-main0} and verify that the sequence admits an interpolating measure.  Lastly, if a sequence begins with zero,  then one need take a new nonzero initial moment and repeat the process used in the above.  
\end{proof}


Before we conclude this note, let us discuss how investigate the location of atoms of an interpolating measure.  In addition, an algorithmic approach to find an explicit formula of a measure will be presented through a concrete example.  
Recall that in the presence of  a (positive) representing measure $\mu$ for  a positive $M_d(n)(\bt)$,  Proposition 3.1 in \cite{tcmp1} states that
$$
\hat p \in \ker M_d(n)(\beta) \iff p({\bf X})=\tb 0 \iff \operatorname{supp}\; \mu  \sbs \mcal Z(p) .
$$
This result provides an evidence that  where the atoms of $\mu$ lie for a singular $M_d(n)$; that is, the algebraic variety of $M_d(n)$ must contain the support of a representing measure. However,  the following example shows such an argument is no longer  valid for the moment problem about an interpolating measure; consider 
\begin{equation}\label{e-m1}
M_2(1)\equiv M_2(1)\left(\beta^{(2)}\right) = \begin{pmatrix}
-1&-16&-4\\
-16&-94&-10\\
-4&-10&2
\end{pmatrix}.
\end{equation} 
Note that $M_2(1)$ has a single column relation 
$X_2=-(4/3) \ti{1}+(1/3)X_1$. Indeed, the sequence can be generated by an interpolating  measure  $\nu= \delta_{(-2,1)}+ \delta_{(-2,-2)} -\delta_{(1,1)}- \delta_{(10,1)}$; but, different from the case for a positive measure, $\supp \nu \not\sbs \mcal Z(x_2+4/3 -(1/3)x_1)=\mcal V_{\bt^{(2)}}$.  In other words, an interpolating measure for the sequence may have atoms outside of the algebraic variety. Nonetheless, one can still  find an interpolating measure supported in the algebraic variety of $M_2(1)$ as follows:

\begin{example} 
We illustrate how to find an interpolating measure of the sequence in (\ref{e-m1}). To find an interpolating measure supported in the algebraic variety of $M_2(1)$, select  a point $\lp a, \frac{a-4}{3}\rp\in \mcal Z(x_2+4/3 -(1/3)x_1)$  for some  $a\in \re$. 
 Using the rank-one decomposition, we write 
\begin{equation}\label{e-rod}
M_2(1) =\widetilde{M_2(1)}+ u \begin{pmatrix}
 1 & a & \frac{a-4}{3} \\
 a & a^2 &\frac{a(a-4)}{3}\\
\frac{a-4}{3}& \frac{a(a-4)}{3}& \frac{(a-4)^2}{9}\\
\end{pmatrix}
\end{equation}
for some $u\in\re$.  
Note that $\rank M_2(1)=2$ and we are attracted to guess that a minimal interpolating measure is 2-atomic (cf. Lemma \ref{r-ran} and \ref{r-rann}).  
In order to find such a  measure, we impose a condition that $\rank \widetilde{M_2(1)} =1$; a calculation shows  $\rank \widetilde{M_2(1)}=1$ if and only if $u=162/(a^2-32 a+94)$.  If we take $u=162/(a^2-32 a+94)$, then 
\begin{eqnarray*}
{M}_2(1)=
\frac{-(a-16)^2}{a^2-32 a+94}
\left(
\begin{array}{ccc}
 1 & \frac{2 (8 a-47)}{a-16} & \frac{2 (2 a-5)}{a-16} \\
 \frac{2 (8 a-47)}{a-16} & \frac{4 (8 a-47)^2}{(a-16)^2} & \frac{4 (2 a-5) (8 a-47)}{(a-16)^2} \\
 \frac{2 (2 a-5)}{a-16} & \frac{4 (2 a-5) (8 a-47)}{(a-16)^2} & \frac{4 (2 a-5)^2}{(a-16)^2} \\
\end{array}
\right)  \\
+\frac{162}{a^2-32 a+94}\begin{pmatrix}
 1 & a & \frac{a-4}{3} \\
 a & a^2 &\frac{a(a-4)}{3}\\
\frac{a-4}{3}& \frac{a(a-4)}{3}& \frac{(a-4)^2}{9}\\
\end{pmatrix}.
\end{eqnarray*}
Therefore, we get an  interpolating measure $\mu= \frac{-(a-16)^2}{a^2-32 a+94}\delta_{\lp \frac{2 (8 a-47)}{a-16} ,  \frac{2 (2 a-5)}{a-16}  \rp}+\frac{162}{a^2-32 a+94}\delta_{\lp a, \frac{a-4}{3} \rp}$ (with $a^2-32 a+94\neq 0$ and $a\neq 16$), which is supported in $\mcal V_{M_2(1)}$.  

\end{example}

\begin{example}\label{example:perturbation}
Consider a truncated moment sequence $\beta^{(4)}$:
\begin{align*}
\ \ &\bt_{00}=6,  &&\bt_{10}=6, &&\bt_{01}=20, &&\bt_{20}=18, &&\bt_{11}=16, \qquad \bt_{02}=68, \\ 
\ \ &\bt_{30}=30, &&\bt_{21}=56,  &&\bt_{12}=40, &&\bt_{03}=236,&&\\
\ \ &\bt_{40}=66, &&\bt_{31}=88, &&\bt_{22}=176,&& \bt_{13}=88, &&\bt_{04}=836.
\end{align*}
Construct its moment matrix as follows: 
\begin{equation*}
\widetilde{M_2(2)}
=
\begin{pmatrix}
6 & 6 & 20 & 18 & 16 & 68 \\
6 & 18 & 16 & 30 & 56 & 40\\
20 & 16 & 68 & 56 & 40 & 236\\
18 & 30 & 56 & 66 & 88 & 176\\
16 & 56 & 40 & 88 & 176 & 88\\
68 & 40 & 236 & 176 & 88 & 836
\end{pmatrix}
\end{equation*}
It is easy to check that the representing measure is $\mu= 2\delta_{(-1,4)} + 4 \delta_{(2,3)}$.
Assume that  
for sufficiently small perturbation we have 
\begin{center}\scalebox{0.88}{$
M_2(2)
=
\begin{pmatrix}
5.990000 &	5.995000 &	19.998000 &17.997500 &	15.999000 & 67.999600\\
5.995000 &	17.997500 &	15.999000 &	29.998750 &	55.999500 &	39.999800\\
19.998000 &	15.999000 & 67.999600 & 	55.999500 &	39.999800 & 235.999920 \\
17.997500 & 29.998750 & 	55.999500 &	65.999375 &	87.999750 &	175.999900\\
15.999000 &	55.999500 &	39.999800 &	87.999750 &	175.999900 &	87.999960\\
67.999600 &	39.999800 & 235.999920 &175.999900 &	87.999960 &	835.999984 
\end{pmatrix}$,}
\end{center}
which is not positive semidefinite. So, arbitrarily small perturbations of a given sequence eject one from the cone of positive semidefinite matrices.  As a result, this sequence does not have a representing measure. 
Instead, one can find  interpolating measures for the sequence.  Concretely, one of them is  $\mu= -0.01\delta_{(0.5,0.2)} + 2\delta_{(-1,4)} + 4 \delta_{(2,3)}$; here the first term with the negative density can be considered as noise. 
\end{example}

Removing noise from the original data is a challenging problem in many different fields. Moment sequences need to be modified since data obtained from physical experiments and phenomena are often corrupt or incomplete. By Theorem \ref{r-main1}, 
one can find an interpolating measure $\mu$ for the given data, which is $\mu= \mu^{+}  -\mu^{-}$ of two positive measures $\mu^{+}$ and $\mu^{-}$. 
Assuming that $\mu^{-}$ is generated by the distribution of noise,  $\mu^{+}$ can be a measure for the denosing data in a sense.

Finally, we will see Boas' result introduced earlier can be extended to any infinite multi-index sequences. It suffices to show that the claim holds for the case of double-index sequences because the argument in the proof also is also valid for any  multi-index sequences.  

\begin{theorem} Any bivariate (real) full moment sequence has an interpolating measure. 
\end{theorem}

\begin{proof} Let $\q \bt_{ij} \w_{i,j=0}^\infty $ be an infinite moment sequence. Then we may write $\bt_{ij} = u_i v_j$ for some $u_i , \ v_j\in \re$. Then, by Boas' result, there are signed measures $\mu$ and $\nu$ supported on $[0,\infty)$ such that 
\begin{align*}
u_i = \int_0^\infty x^i \, d\mu \qquad  \text{and} \qquad  v_j = \int_0^\infty y^j \, d\nu.
\end{align*}
It follows from the Jordan decomposition theorem that any signed measure can be written as a difference of two positive measure; that is, 
\begin{align*}
\mu = \mu_+ - \mu_- \qquad  \text{and} \qquad  \nu = \nu_+ - \nu_-,
\end{align*}
where   $\mu_+$, $ \mu_-$, $\nu_+ $, and $\nu_-$ are positive measures supported on  $[0,\infty)$. Now, 
\begin{align*}
\bt_{ij} = u_i  v_j &=  \int_0^\infty x^i \, d\mu  \int_0^\infty y^j \, d\nu \\
&= \lb  \int_0^\infty x^i \, d( \mu_+ - \mu_-)  \rb \lb \int_0^\infty y^j \, d( \nu_+ - \nu_-) \rb\\
&= \lp  \int_0^\infty x^i \, d \mu_+ - \int_0^\infty x^i \, d\mu_-  \rp  \lp \int_0^\infty y^j \, d \nu_+ -\int_0^\infty y^j \, d \nu_- \rp\\ 
&= \lp \int_0^\infty x^i \, d \mu_+ \int_0^\infty y^j \, d \nu_+ +  \int_0^\infty x^i \, d \mu_- \int_0^\infty y^j \, d \nu_- \rp  \\
& \qquad \qquad  - \lp \int_0^\infty x^i \, d \mu_+ \int_0^\infty y^j \, d \nu_- +  \int_0^\infty x^i \, d \mu_- \int_0^\infty y^j \, d \nu_+ \rp   \\
&=\int_0^\infty  \int_0^\infty  x^i  y^j \, d( \mu_+ \times \nu_+)  +  \int_0^\infty  \int_0^\infty  x^i  y^j \, d( \mu_- \times \nu_-)  \\
& \qquad \qquad  - \int_0^\infty  \int_0^\infty  x^i  y^j \, d( \mu_+ \times \nu_-) - \int_0^\infty  \int_0^\infty  x^i  y^j \, d( \mu_- \times \nu_+)\\
&=\int_0^\infty  \int_0^\infty  x^i  y^j \, d \tau ,
  \end{align*}
where $\tau :=  \mu_+ \times \nu_+ +  \mu_- \times \nu_- -  \mu_+ \times \nu_- -  \mu_- \times \nu_+$ is a signed measure. Note that the second last identity in the above is true since $\mu$ and $\nu$ satisfy the hypothesis in the Fubini's theorem. Indeed, observe that $u_0 =  \int_0^\infty  \, d\mu = \mu([0,\infty)) \in \re   $ and $v_0 =  \int_0^\infty  \, d\nu = \nu([0,\infty)) \in \re  $; thus, $\mu$ and $\nu$ are finite measures, which means the two measure are $\sigma$-finite. Also, $x^i y^j$ are nonnegative on $[0,\infty)\times [0,\infty)$ for any $i, j \in \nn_0$.  
Therefore,  $\q \bt_{ij} \w$  has an interpolating measure $\tau.$ 
\end{proof}

\bsk

\ti{Acknowledgment.} 
The authors are indebted to Professor Ilwoo Cho and Professor Ra\'ul Curto for several discussions that lead a better presentation of this note.    



\end{document}